\documentclass[10pt]{amsart}
\usepackage{verbatim}
\usepackage{eucal,url,amssymb,stmaryrd,enumerate,amscd}
\usepackage{amsmath}
\usepackage[pagebackref,colorlinks=true]{hyperref}
\usepackage[pagebackref]{hyperref}
\usepackage{amsfonts}
\usepackage{amsthm}
\usepackage[margin=.9in]{geometry}

\numberwithin{equation}{section}
\newtheorem{thrm}{Theorem}[section]
\newtheorem{lemma}[thrm]{Lemma}
\newtheorem{prop}[thrm]{Proposition}
\newtheorem{cor}[thrm]{Corollary}
\newtheorem{dfn}[thrm]{Definition}

\newtheorem{rmrk}[thrm]{Remark}

\newtheorem{conv}[thrm]{Convention}
\newtheorem{conj}[thrm]{Conjecture}

\newcommand{\vol}{\, Vol_{\theta}}

\def\gr{\nabla f}
\def\bi{\nabla}

\begin{document}

\begin{abstract}
We prove a CR Obata type result that if the first
positive eigenvalue of the sub-Laplacian on a compact strictly pseudoconvex
pseudohermitian manifold with a divergence free pseudohermitian torsion
takes the smallest possible value then, up to a homothety of the
pseudohermitian structure, the manifold is the standart Sasakian unit sphere.  { We also give a  version of this theorem using the existence of a function with traceless  horizontal Hessian  on a complete, with respect to Webster's metric, pseudohermitian manifold. }
\end{abstract}

\title[An Obata type result for the first eigenvalue of the sub-Laplacian on
a CR manifold]{An Obata type result for the first eigenvalue of the
sub-Laplacian on a CR manifold with a divergence free torsion}
\date{\today }
\author{S. Ivanov}
\address[Stefan Ivanov]{University of Sofia, Faculty of Mathematics and
Informatics, blvd. James Bourchier 5, 1164, Sofia, Bulgaria}
\email{ivanovsp@fmi.uni-sofia.bg}
\author{D. Vassilev}
\address[Dimiter Vassilev]{ Department of Mathematics and Statistics\\
University of New Mexico\\
Albuquerque, New Mexico, 87131-0001}
\email{vassilev@math.unm.edu}
\maketitle
\tableofcontents


\setcounter{tocdepth}{2}

\section{Introduction}

The classical theorems of Lichnerowicz \cite{Li} and Obata \cite{O3} give
correspondingly a lower bound for the first eigenvalue of the Laplacian on a
compact manifold with a lower Ricci bound and characterize the case of
equality. In \cite{Li} it was shown that for every compact Riemannian
manifold $(M,h)$ of dimension $n$ for which the Ricci curvature is greater than or
equal to that of the round unit $n$-dimensional sphere $S^n(1)$, i.e., $
Ric(X,Y)\geq (n-1)h(X,Y),$
we have that the first positive eigenvalue $\lambda_1$ of the (positive)
Laplace operator is greater than or equal to the first eigenvalue of the
sphere, $\lambda_1\geq n.$

Subsequently in \cite{O3} it was shown that the lower bound for
the eigenvalue is achieved iff the Riemannian manifold
is isometric to $S^n(1)$. Lichnerowicz proved his result using the
classical Bochner-Weitzenb\"ock formula. In turn, Obata showed
that under these assumptions the trace-free part of the Riemannian Hessian of an eigenfunction $f$ with eigenvalue $\lambda_1=n$ vanishes,
\begin{equation}\label{clasob}D^2 f = -f h,
\end{equation}
after which he defined an isometry using analysis based on
the geodesics and
Hessian comparison of the distance function from a point. More precisely, Obata showed in \cite{O3} that if on a complete Riemannian manifold there exists a non-constant function satisfying \eqref{clasob} then the manifold is isometric to the unit sphere. Later Gallot \cite%
{Gallot79} generalized these results to statements involving the higher
eigenvalues and corresponding eigenfunctions of the Laplace operator.

From the sub-ellipticity of the sub-Laplacian defined in many well
studied sub-Riemannian geometries it follows that its spectrum is
discrete on a compact manifold. It is therefore natural to ask if
there is a sub-Riemannian version of the above results. In fact, a CR analogue of the Lichnerowiecz theorem was found by
Greenleaf \cite{Gr} for dimensions $2n+1>5$,  while the corresponding results for  $n=2$ and  $n=1$ were achieved later in \cite{LL}  and
\cite{Chi06}, respectively. As a continuation of this line of results in the setting of geometries modeled on the rank  one  symmetric spaces in \cite{IPV1} it was proven a  quaternionic contact version of the Lichnerowicz result.

{ The CR Lichnerowicz type result states that on a compact $2n+1$-dimensional strictly pseudoconvex pseudohermitian manifold satisfying a certain positivity condition the first eigenvalue of the sub-laplacian is grater or equal to that of the standard Sasakian sphere}.
For the exact statement of  the CR Lichnerowicz type result we refer the reader to Theorem \ref{main1}. For ease of reference we also include complete proofs of the known
results in the CR case. The presented proof of Theorem \ref{main1} uses the known techniques from  \cite{Gr}, \cite{LL}, \cite{Chi06}, but is based  solely on the non-negativity of the Paneitz operator thereby slightly simplifing the known arguments.
Greenleaf \cite{Gr} showed the result for $n\geq 3$, while S.-Y. Li and
H.-S. Luk adapted Greenleaf's prove to cover the case $n=2$. They also gave
a version of the case $n=1$ assuming further a condition on the covariant
derivative with respect to the Tanaka-Webster connection of the
pseudohermitian torsion tensor. Part b) in Theorem~\ref{main1} was
established by H.-L. Chiu in \cite{Chi06}. We remark that if $n>1$ the Paneitz operator is always non-negative, cf. Lemma \ref{l:GrLee} {  while  in the case $n=1$ the vanishing of the pseudohermitian torsion implies that the Paneitz operator is non-negative,  see \cite{Chi06} and\cite{CCC07}.}

Other relevant for this paper results in the CR case have been
proved in  \cite{CC07,CC09b,CC09a}, \cite{Bar} and \cite{ChW}
adding a corresponding inequality for $n=1$, or characterizing the equality
case in the vanishing pseudohermitian torsion case (the Sasakian case).

The problem of the existence of an Obata-type theorem in pseudohermitian
manifold was considered in \cite{CC09a} where the following CR analogue of
Obata's theorem was conjectured.

\begin{conj}[\protect\cite{CC09a}]
\label{conj1} Let $(M,\theta)$ be a closed pseudohermitian (2n+1)-manifold
with $n\ge 2$. In addition we assume the Paneitz operator is nonnegative if
n = 1. Suppose there is a positive constant $k_0$ such that the
pseudohermitian Ricci curvature $Ric$ and the pseudohermitian torsion $A$
satisfy the inequality \eqref{condm}. If $\frac{n}{n+1}k_0$ is an eigenvalue
of the sub-Laplacian then $(M,\theta)$ is the standard (Sasakian) CR
structure on the unit sphere in $\mathbb{C}^{n+1}$. %
\end{conj}
This conjecture was proved in the case of vanishing pseudohermitian torsion
(Sasakian case) in \cite{CC09a} for $n\ge 2$ and in \cite{CC09b} for $n=1$.

The non-Sasakian case was also considered in \cite{ChW12} where the
Conjecture~\ref{conj1} was established under the following assumptions on
the pseudohermitian torsion (in complex coordinates):
\begin{align*}
& \text{ (i) for } \quad n\geq 2, \qquad A_{\alpha\beta,\, \bar
\beta}=0,\quad \text{ and } \quad A_{\alpha\beta,\, \gamma\bar \gamma}=0,
\quad \text{\cite[Theorem~1.3]{ChW12} }; \\
& \text{ (ii) for } \quad n=1, \qquad A_{11,\, \bar 1}=0,\quad \text{ and }
\quad P_1 f=0 \quad \text{\cite[Theorem~1.4]{ChW12}},
\end{align*}
where $$P_\alpha f=f_{\bar \beta} {^{\bar \beta}} {_{\alpha}}+inA_{\alpha\beta}f^{\beta}$$ is the operator
characterizing CR-pluriharmonic functions when $n=1$, see also the paragraph after Remark \ref{r:non-negative paneitz}. The first condition, $A_{\alpha\beta,\, \bar
\beta}=0$, means that the (horizontal real) divergence of $A$
vanishes,
\begin{equation*}
(\nabla^*A)\,(X)=-(\nabla_{e_a} A)\, (e_a,X)=0.
\end{equation*}

{One purpose of this paper is to establish Conjecture~\ref{conj1} in the
(non-Sasakian) case of a divergence-free pseudohermitian torsion where we prove the following result.}
\begin{thrm}\label{main2} Let $(M, \theta)$ be a compact strictly pseudoconvex
pseudohermitian CR manifold of dimension $2n+1$. Suppose there is a positive
constant $k_0$ such that the pseudohermitian Ricci curvature $Ric$ and the
pseudohermitian torsion $A$ satisfy the inequality
\begin{equation}  \label{condm}
Ric(X,X)+ 4A(X,JX)\geq k_0 g(X,X).
\end{equation}
Furthermore, suppose the \textit{horizontal} divergence of the pseudohermitian
torsion vanishes,
\begin{equation*}  
\nabla^* A=0.
\end{equation*}

\begin{itemize}
\item[a)] If $n\geq 2$ and $\lambda=\frac{n}{n+1}k_0$ is an eigenvalue of
the sub-Laplacian, 
then up-to a scaling of $\theta$ by a positive constant $(M,\theta)$ is the
standard (Sasakian) CR structure on the unit sphere in $\mathbb{C}^{n+1}$.

\item[ b)] If $n=1$ and $\lambda=\frac{1}{2}k_0$ is an eigenvalue of the
sub-Laplacian, 
the same conclusion can be reached assuming in addition that the Paneitz
operator is non-negative, i.e., for a smooth function $f$
\begin{equation*}
-\int_M P_{f}(\nabla f)Vol_{\theta}\geq 0.
\end{equation*}
\end{itemize}
\end{thrm}

The value of the scaling is determined, for example, by the fact that the
standard psudohermitian structure on the unit sphere has first eigenvalue
equal to $2n$. The corresponding eigenspace is spanned by the restrictions
of all linear functions to the sphere.

{Our approach is based on Lemma~\ref{l:hessian in extremal case}  where we find the explicit form of the Hessian with respect to the Tanaka-Webster connection of an extremal
eigenfunction $f$, i.e., an eigenfunction with eigenvalue $n/(n+1)k_0$, and the formula for the pseudohermitian curvature. As mentioned earlier a proof of Greenleaf's result based on the non-negativity of the Paneitz operator can be found in the Appendix. This proof shows formula  \eqref{eq7} for the horizontal Hessian of $f$, which after a rescaling can be put in the form }
\begin{equation*}
 \nabla ^{2}f(X,Y)=-fg(X,Y)-df(\xi )\omega (X,Y), \qquad X, Y \in H=Ker\, \theta.
\end{equation*}
{We prove Theorem~\ref{main2} as a consequence of Theorem~\ref{t:A vansihes if div} and Theorem~\ref{t:A vansihes if div 3D} taking into account the already established CR Obata theorem  for pseudohermitian manifold with a vanishing pseudohermitian torsion.  Thus, the key new results are Theorem \ref{t:A vansihes if div} and Theorem \ref{t:A vansihes if div 3D} which show correspondingly in the case $n\geq 2$ and $n=1$ that if the pseudohermitian torsion is divergence-free and we have an eigenfunction $f$ with a horizontal Hessian given with the above formula then the pseudohermitian  torsion vanishes, i.e., we have a Sasakian structure.}

{The local nature of the
analysis leading to the proof of Theorem \ref{main2} allows to prove our second main result, which is the following CR-version of Obata's Theorem \cite{O3}.}
\begin{thrm}\label{main11} Let $(M, \theta)$ be a  strictly pseudoconvex
pseudohermitian CR manifold of dimension $2n+1\ge 5$ with a
divergence-free  pseudohermitian torsion, $\bi^*A=0$. Assume, further, that $M$ is
complete with respect to the Rieemannian metric $h=g+\theta^2$. If
there is a smooth function $f\not\equiv 0$ whose Hessian with respect
to the Tanaka-Webster connection satisfies
\begin{equation}\label{e:hessian}
 \nabla ^{2}f(X,Y)=-fg(X,Y)-df(\xi )\omega (X,Y), \qquad X, Y \in H=Ker\, \theta,
\end{equation}
then up to a scaling of $\theta$ by a positive constant
$(M,\theta)$ is the standard (Sasakian) CR structure on the unit
sphere in $\mathbb{C}^{n+1}$.

{In dimension three
  the above result holds provided the pseudohermitian torsion vansihes, $A=0$.}
\end{thrm}

We finish the introduction by recalling that Lichnerowicz' type results (not in sharp forms) in a
general sub-Riemannian setting were shown in \cite{Bau2}, \cite{Bau3} (these
two papers apply only to the vanishing pseudohermitian torsion CR case) and
\cite{Hla}. It will be interesting to consider whether the results of \cite{Bau2}, \cite%
{Bau3}, \cite{ChTW10} and \cite{ChW} can be extended to the non-Sasakian,
but divergence free pseudohermitian torsion case.

\begin{conv}
\label{conven} \hfill\break\vspace{-15pt}

\begin{enumerate}
\item[a)] We shall use $X,Y,Z,U$ to denote horizontal vector fields, i.e. $%
X,Y,Z,U\in H=Ker\, \theta$.

\item[b)] $\{e_1,\dots,e_{2n}\}$ denotes a local orthonormal basis of the
horizontal space $H$.

\item[c)] The summation convention over repeated vectors from the basis $%
\{e_1,\dots,e_{2n}\}$ will be used. For example, for a (0,4)-tensor $P$, the
formula $k=P(e_b,e_a,e_a,e_b)$ means
\begin{equation*}
k=\sum_{a,b=1}^{2n}P(e_b,e_a,e_a,e_b);
\end{equation*}
\end{enumerate}
\end{conv}

\textbf{Acknowledgments} The authors would like to thank Alexander Petkov for pointing to us that Lemma \ref{gr3} holds for all $n$ as a simple consequence of the Ricci identity thus simplifying our initial approach for the special case $n=1$. The research is partially supported by Contract
``Idei", DO 02-257/18.12.2008 and Contract ``Idei", DID 02-39/21.12.2009.
S.I is partially supported by Contract 181/2011 with the University of Sofia
`St.Kl.Ohridski'

\section{Pseudohermitian manifolds and the Tanaka-Webster connection}

In this section we will briefly review the basic notions of the
pseudohermitian geometry of a CR manifold. Also, we recall some results (in
their real form) from \cite{T,W,W1,L1}, see also \cite{DT,IVZ,IV2}, which we
will use in this paper.

A CR manifold is a smooth manifold $M$ of real dimension 2n+1, with a fixed
n-dimensional complex sub-bundle $\mathcal{H}$ of the complexified tangent
bundle $\mathbb{C}TM$ satisfying $\mathcal{H} \cap \overline{\mathcal{H}}=0$
and $[ \mathcal{H},\mathcal{H}]\subset \mathcal{H}$. If we let $H=Re\,
\mathcal{H}\oplus\overline{\mathcal{H}}$, the real sub-bundle $H$ is
equipped with a formally integrable almost complex structure $J$. We assume
that $M$ is oriented and there exists a globally defined compatible contact
form $\theta$ such that the \emph{horizontal space} is given by $${H}=Ker\,\theta.$$ In other words, the hermitian
bilinear form
\begin{equation*}
2g(X,Y)=-d\theta(JX,Y)
\end{equation*}
is non-degenerate. The CR structure is called strictly pseudoconvex if $g$
is a positive definite tensor on $H$. The vector field $\xi$ dual to $\theta$
with respect to $g $ satisfying $\xi\lrcorner d\theta=0$ is called the Reeb
vector field. The almost complex structure $J$ is formally integrable in the
sense that
\begin{equation*}
([JX,Y]+[X,JY])\in {H}
\end{equation*}
and the Nijenhuis tensor
\begin{equation*}
N^J(X,Y)=[JX,JY]-[X,Y]-J[JX,Y]-J[X,JY]=0.
\end{equation*}
A CR manifold $(M,\theta,g)$ with a fixed compatible contact form $\theta$
is called \emph{a pseudohermitian manifold}%
\index{pseudohermitian manifold}. In this case the 2-form
\begin{equation*}
d\theta_{|_{{H}}}:=2\omega
\end{equation*}
is called the fundamental form. Note that the contact form is determined up
to a conformal factor, i.e. $\bar\theta=\nu\theta$ for a positive smooth
function $\nu$, defines another pseudohermitian structure called
pseudo-conformal to the original one.

\subsection{Invariant decompositions}

As usual any endomorphism $\Psi$ of $H$ can be decomposed with respect to
the complex structure $J$ uniquely into its $U(n)$-invariant $(2,0)+(0,2)$
and $(1,1)$ parts. In short we will denote these components correspondingly
by $\Psi_{[-1]}$ and $\Psi_{[1]}$. Furthermore, we shall use the same
notation for the corresponding two tensor, $\Psi(X,Y)=g(\Psi X,Y)$.
Explicitly, $\Psi=\Psi=\Psi_{[1]}+\Psi_{[-1]}$, where
\begin{equation}
{\label{comp}} \Psi_{[1]}(X,Y)=%
\frac {1}{2}\left [ \Psi(X,Y)+\Psi(JX,JY)\right ], \qquad \Psi_{[-1]}(X,Y)=%
\frac {1}{2}\left [ \Psi(X,Y)-\Psi(JX,JY)\right ].
\end{equation}
The above notation is justified by the fact that the $(2,0)+(0,2)$ and $%
(1,1) $ components are the projections on the eigenspaces of the operator
\begin{equation*} 
\Upsilon =\ J\otimes J, \quad (\Upsilon \Psi) (X,Y)\overset{def}{=}\Psi
(JX,JY),
\end{equation*}
corresponding, respectively, to the eigenvalues $-1$ and $1$. Note that both
the metric $g$ and the 2-form $\omega$ belong to the [1]-component, since $%
g(X,Y)=g (JX,JY)$ and $\omega(X,Y)=\omega (JX,JY)$. Furthermore, the two
components are orthogonal to each other with respect to $g$.

\subsection{The Tanaka-Webster connection}

The Tanaka-Webster connection \cite{T,W,W1} is the unique linear connection $%
\nabla$ with torsion $T$ preserving a given pseudohermitian
structure, i.e., it has the properties
\begin{equation}  \label{torha}
\begin{aligned}& \nabla\xi=\nabla J=\nabla\theta=\nabla g=0,\\
& T(X,Y)=d\theta(X,Y)\xi=2\omega(X,Y)\xi, \quad T(\xi,X)\in {H},
\\ & g(T(\xi,X),Y)=g(T(\xi,Y),X)=-g(T(\xi,JX),JY). \end{aligned}
\end{equation}
Let $f$ be a smooth function on a pseudohermitian manifold $M$ with $\nabla
f $ its horizontal gradient, $g(\nabla f,X)=df(X)$.
The horizontal sub-Laplacian $\triangle f$ and the norm of the horizontal
gradient $\nabla f =df(e_a)e_a$ of a smooth function $f$ on $M$ are defined
respectively by
\begin{equation}  \label{lap}
\triangle f\ =-\ tr^g_H(\nabla df)\ =\nabla^*df= -\ \nabla df(e_a,e_a),
\qquad |\nabla f|^2\ =\ df(e_a)\,df(e_a).
\end{equation}
The function $f\not\equiv 0$ is an eigenfunction of the sub-Laplacian if
\begin{equation}  \label{eig}
\triangle f =\lambda f,
\end{equation}
where $\lambda$ is a (necessarily, non-negative,) constant.

It is well known that the endomorphism $T(\xi,.)$ is the obstruction a
pseudohermitian manifold to be Sasakian. The symmetric endomorphism $T_\xi:{H%
}\longrightarrow {H}$ is denoted by $A$, $A(X,Y):=T(\xi,X,Y)$,  and it is call \emph{the torsion of
the pseudohermitian manifold or pseudohermitian torsion.} The
pseudohermitian torsion $A$ is a completely trace-free tensor of type (2,0)+(0,2),
\begin{equation}  \label{tortrace}
A(e_a,e_a)=A(e_a,Je_a)=0, \quad A(X,Y)=A(Y,X)=-A(JX,JY).
\end{equation}
Let $R$ be the curvature of the Tanaka-Webster connection. The
pseudohermitian Ricci tensor $Ric$, the pseudohermitian scalar curvature $S$
and the pseudohermitian Ricci 2-form $\rho$ are defined by
\begin{equation*}
Ric(A,B)=R(e_a,A,B,e_a), \quad S=Ric(e_a,e_a),\quad
\rho(A,B)=\frac12R(A,B,e_a,Ie_a).
\end{equation*}
We summarize below the well known properties of the curvature $R$ of the
Tanaka-Webster connection \cite{W,W1,L1} using real expression, see also
\cite{DT,IVZ,IV2}.
\begin{equation*} 
R(X,Y,JZ,JV)=R(X,Y,Z,V)=-R(X,Y,V,Z), \qquad R(X,Y,Z,\xi)=0,
\end{equation*}
\begin{multline}  \label{currrr}
\frac12\Big[R(X,Y,Z,V)-R(JX,JY,Z,V)\Big] %
=-g(X,Z)A(Y,JV)-g(Y,V)A(X,JZ)+g(Y,Z)A(X,JV) \\
+g(X,V)A(Y,JZ)
-\omega(X,Z)A(Y,V)-\omega(Y,V)A(X,Z)+\omega(Y,Z)A(X,V)+\omega(X,V)A(Y,Z),
\end{multline}
\begin{equation}  \label{currr}
R(\xi,X,Y,Z)=(\nabla_YA)(Z,X)-(\nabla_ZA)(Y,X),\hskip1truein
\end{equation}
\begin{equation}  \label{torric}
\begin{aligned} & Ric(X,Y)=Ric(Y,X),\hskip2.5truein\\ &
Ric(X,Y)-Ric(JX,JY)=4(n-1)A(X,JY), \end{aligned}
\end{equation}
\begin{equation}  \label{rho}
2\rho(X,JY)=-Ric(X,Y)-Ric(JX,JY)=R(e_a,Je_a,X,JY),
\end{equation}
\begin{equation}  \label{div}
2(\nabla_{e_a}Ric)(e_a,X)= dS(X).\hskip2.5truein
\end{equation}
The equalities \eqref{torric} and \eqref{rho} imply
\begin{equation}  \label{rid}
Ric(X,Y)=\rho(JX,Y)+2(n-1)A(JX,Y),
\end{equation}
i.e. $\rho$ is the $(1,1)$-part of the pseudohermitian Ricci tensor  while
the $(2,0)+(0,2)$-part is given by the pseudohermitian torsion $A$.

\subsection{The Ricci identities for the Tanaka-Webster connection}

We shall use repeatedly the following Ricci identities of order two and
three for a smooth function $f$, see also \cite{IV2},
\begin{equation}  \label{e:ricci identities}
\begin{aligned} & \nabla^2f (X,Y)-\nabla^2f(Y,X)=-2\omega(X,Y)df(\xi)\\ &
\nabla^2f (X,\xi)-\nabla^2f(\xi,X)=A(X,\nabla f)\\ & \nabla^3 f
(X,Y,Z)-\nabla^3 f(Y,X,Z)=-R(X,Y,Z,\nabla f) - 2\omega(X,Y)\nabla^2f
(\xi,Z)\\ &\nabla ^{3}f(X,Y,Z)-\nabla ^{3}f(Z,Y,X)=-R(X,Y,Z,\nabla
f)-R(Y,Z,X,\nabla f)-2\omega (X,Y)\nabla ^{2}f(\xi ,\ Z)\\ &\hskip 1.8in
-2\omega (Y,Z)\nabla ^{2}f(\xi ,X) +2\omega (Z,X)\nabla ^{2}f(\xi ,Y)
+2\omega (Z,X)A(Y,\nabla f) \\ & \nabla^3 f(\xi,X,Y)-\nabla^3
f(X,\xi,Y)=(\nabla_{\gr}A)(Y,X)-(\nabla_YA)(\nabla f,X)-\nabla^2 f(AX,Y)\\
&\nabla^3 f(X,Y,\xi)-\nabla^ 3 f (\xi,X,Y)=\nabla^2f (AX,Y)+\nabla^2f (X,AY)
+(\nabla_X A)(Y,\nabla f)+(\nabla_Y A)(X,\nabla f)\\ &\hskip4.4in
-(\nabla_{\nabla f}) A( X,Y). \end{aligned}
\end{equation}
 We note that the above Ricci identities for the Tanaka-Webster connection follow from the general Ricci identities for a connection with torsion applying the properties of the pseudohermitian torsion listed in \eqref{torha} and the curvature identity \eqref{currr}.  For example,
\begin{multline*}
\nabla^3 f(\xi,X,Y)-\nabla^3 f(X,\xi,Y)=-R(\xi,X,Y,\nabla f)-\nabla^2
f(T(\xi,X),Y) \\
=R(X,\xi,Y,\nabla f)-A(X,e_a)\nabla^2 f(e_a,Y) =(\nabla_{\nabla
f}A)(Y,X)-(\nabla_YA)(\nabla f,X)-A(X,e_a)\nabla^2 f(e_a,Y),
\end{multline*}
where we used  \eqref{currr}.

An important consequence of  the first Ricci identity is the following fundamental formula
\begin{equation}  \label{xi1}
g(\nabla^2f,\omega)=\nabla^2f(e_a,Je_a)=-2n\,df(\xi).
\end{equation}
On the other hand, by \eqref{lap} the trace with respect to the metric is the negative sub-Laplacian
\[
g(\nabla^2f,g)=\nabla^2f(e_a,e_a)=-\triangle f.
\]

We also recall the horizontal divergence theorem \cite{T}. Let $(M,
g,\theta) $ be a pseudohermitian manifold of dimension $2n+1$. For a fixed
local 1-form $\theta$ the form
\begin{equation*} 
Vol_{\theta}=\theta\wedge\omega^{n}
\end{equation*}
is a globally defined volume form since $Vol_{\theta}$ is independent on the
local one form $\theta$.

We define the (horizontal) divergence of a horizontal vector
field/one-form $\sigma\in\Lambda^1\, (H)$ defined by
\begin{equation*}
\nabla^*\, \sigma\ =-tr|_{H}\nabla\sigma=\ -(\nabla_{e_a}\sigma)e_a.
\end{equation*}
The following Proposition, which allows "integration by parts", is well
known \cite{T}.

\begin{prop}
\label{div1} On a compact pseudohermitian manifold $M$ the following
divergence formula holds true
\begin{equation*}
\int_M (\nabla^*\sigma)Vol_{\theta}\ =\ 0.
\end{equation*}
\end{prop}

\section{The hessian of an extremal function in the extremal case. }

Our goal is to determine the full Hessian of an "\textit{extremal
first eigenfunction}" which is an eigenfunction with the smallest possible in the sense of  Theorem \ref{main1} eigenvalue.

\begin{lemma}
\label{l:hessian in extremal case} Let $M$ be a compact strictly
pseudoconvex CR manifold of dimension $2n+1$, $n\geq 1$ satisfying
\begin{equation*}
Ric(X,X)+4A(X,JX)=\rho(JX,Y)+2(n+1)A(JX,Y)\geq k_0 \,g(X,Y)
\end{equation*}
while if $n=1$ assume, further, that the Paneitz operator is non-negative on $f$, i.e., \eqref{e:nonnegativeP} holds true.

If $\frac{n}{n+1}k_0$ is an eigenvalue of the sub-Laplacian, then the
corresponding eigenfunctions satisfy the identity
\begin{equation}  \label{eq7}
\nabla^2f(X,Y)=-\frac{k_0}{2(n+1)}fg(X,Y)-df(\xi)\omega(X,Y).
\end{equation}
\end{lemma}

\begin{proof}
Under the assumptions of the Lemma,
inequality \eqref{e:obata ineq} becomes an equality. Therefore,
\begin{equation*}
(\nabla^2f)_{[-1]}=0.
\end{equation*}
In addition, we  must have equality in \eqref{coshy3} hence
\eqref{equality hessian} follows. Thus, the following identity holds true
\begin{equation*}
\nabla^2f(X,Y)= -\frac {1}{2n}(\triangle f)\cdot g (X,Y)+\frac {1}{2n}g
(\nabla^2f,\omega)\cdot\omega (X,Y).
\end{equation*}
Now, taking into account that $f$ is an extremal first eigenfunction we obtain
the coefficient in front of the metric. Finally, the skew-symmetric part of
the horizontal Hessian is determined by the first Ricci identity in
\eqref{e:ricci identities}.
\end{proof}

\begin{rmrk}
In addition to the above identities we have trivially from the proof of Theorem \ref{main1} the next equations
\begin{equation}  \label{eq14}
Ric(\nabla f,\nabla f)+4A(J\nabla f,\nabla f)=k_0|\nabla f|^2, \hskip.7in \int_M P_f(\gr)\vol=0.
\end{equation}
\end{rmrk}

Using a homothety we can reduce to the case $\lambda_1=2n$ and $k_0=2(n+1)$,
which are the values for the standard Sasakian round sphere. Henceforth, we
shall work under these assumptions. Thus, for  an extremal first eigenfunction $f$ (by definition $%
f\not\equiv 0$) and $n\geq 1$  we have the equalities
\begin{equation*} 
\begin{aligned} \lambda = 2n, \qquad \triangle f=2n f, \qquad
\int_M(\triangle f)^2Vol_{\theta}=2n\int_M|\nabla f|^2Vol_{\theta}.
\end{aligned}
\end{equation*}
In addition, the horizontal Hessian of $f$ satisfies \eqref{eq7}, which with the assumed normalization takes the form given in equation \eqref{e:hessian}.

\subsection{The divergence-free torsion and the vertical derivative of an extremal function} Here we show one of our main observations that the vertical derivative of an extremal function is again an extremal function provided the pseudohermitian torsion is divergence-free. We begin with an identity satisfied by every extremal eigenfunction.
\begin{lemma}
\label{l:D3f bis} Let $M$ be a strictly pseudoconvex pseudohermitian CR
manifold of dimension $2n+1\geq 3$. If $f$ is an eigenfunction of the
sub-Laplacian satisfying \eqref{e:hessian}, then the following formula for
the third covariant derivative holds true
\begin{multline}
\nabla ^{3}f(X,Y,\xi )=-df(\xi )g(X,Y)-(\xi ^{2}f)\omega (X,Y)-2fA(X,Y)
\label{e:D3f extremal bis} \\
+(\nabla_X A)(Y,\nabla f)+(\nabla_Y A)(X,\nabla f)-(\nabla_{\nabla f}
A)(X,Y).
\end{multline}
\end{lemma}

\begin{proof}
We start by substituting the horizontal Hessian of $f$ given in %
\eqref{e:hessian} in the forth Ricci identity of \eqref{e:ricci
identities}.
\begin{multline*}
\nabla ^{3}f(X,Y,\xi )=\nabla ^{3}f(\xi ,X,Y)+\nabla ^{2}f(AX,Y)+\nabla
^{2}f(X,AY)+(\nabla _{X}A)(Y,\nabla f)+(\nabla _{Y}A)(X,\nabla f) -(\nabla
_{\nabla f}A)(X,Y) \\
=\nabla ^{3}f(\xi ,X,Y)-fA(X,Y)+df(\xi )A(X,JY)-fA(X,Y)-df(\xi )A(JX,Y) \\
+(\nabla _{X}A)(Y,\nabla f)+(\nabla _{Y}A)(X,\nabla f)-(\nabla _{\nabla
f}A)(X,Y) \\
=\nabla ^{3}f(\xi ,X,Y)-2fA(X,Y)+(\nabla _{X}A)(Y,\nabla f)+(\nabla
_{Y}A)(X,\nabla f)-(\nabla _{\nabla f}A)(X,Y).
\end{multline*}
The first term, $\nabla ^{3}f(\xi ,X,Y),$ in the left-hand side is computed
by differentiating the formula for the horizontal Hessian, \eqref{e:hessian}. A substitution of
the thus obtained formula in the one above gives the desired \eqref{e:D3f extremal bis}
which completes the proof.
\end{proof}
With the help of  Lemma \ref{l:D3f bis} we turn to our main result in this sub-section.
\begin{lemma}\label{c:xi f}
Let $M$ be a strictly pseudoconvex pseudohermitian CR
manifold of dimension $2n+1\geq 3$. If the pseudohermitian torsion is
divergence-free with respect to the Tanaka-Webster connection, $(\nabla
_{e_{a}}A)(e_{a},X)=0$, and $f$ is an eigenfunction satisfying %
\eqref{e:hessian} then the function $\xi f$ is an eigenfunction with the
same eigenvalue,
\begin{equation}
\triangle (\xi f)=2n(\xi f).  \label{xii1}
\end{equation}%
{In particular, if $M$ is compact satisfying \eqref{condm} then the horizontal Hessian of $\xi f$ is given by}
\begin{equation}
\nabla ^{2}(\xi f)(X,Y)=\nabla ^{3}f(X,Y,\xi )=- df(\xi )g(X,Y)-(\xi^2
f)\omega (X,Y).  \label{xii4}
\end{equation}
\end{lemma}

\begin{proof}
From the last Ricci identity in \eqref{e:ricci identities} we have
\begin{equation*}
\triangle (\xi f)-\xi(\triangle f)=\nabla ^{3}f(e_{a},e_{a},\xi )-\nabla ^{3}f(\xi ,e_{a},e_{a}) =2g(A,\nabla
^{2}f)-2(\nabla ^{\ast }A)(\nabla f)+\nabla A(\nabla f,e_{a},e_{a})=0,
\end{equation*}
using that the torsion is trace- and divergence- free, and the fact that $g(A,\nabla
^{2}f)=0 $ by \eqref{e:hessian}. Hence, \eqref{xii1} holds.

{The second part follows from the just proved \eqref{xii1} and Lemma~\ref{l:hessian in extremal case}.}
\end{proof}

\begin{rmrk}
\label{remn=1} In the compact case, the above lemma can also be seen with
the help of the following "vertical Bochner formula" valid for any smooth
function $f$
\begin{equation}  \label{e:vertical Bochner}
-\triangle (\xi f)^2 = 2|\nabla(\xi f)|^2-2df(\xi)\cdot \xi (\triangle f) +
4df(\xi)\cdot g(A,\nabla^2 f) -4df (\xi)(\nabla^*A)(\nabla f).
\end{equation}
However, the argument in Lemma~\ref{c:xi f} is purely
local.

To prove \eqref{e:vertical Bochner}  we use the last of the Ricci
identities \eqref{e:ricci identities} and the fact that the
torsion is trace free to obtain
\begin{eqnarray*}
-\frac{1}{2}\triangle (\xi f)^{2} &=&\nabla ^{3}f(e_{a},e_{a},\xi )df(\xi
)+\nabla ^{2}f(e_{a},\xi )\nabla ^{2}f(e_{a},\xi ) \\
&=&\left[ \nabla ^{3}f(\xi ,e_{a},e_{a})+2g(\nabla ^{2}f,A)-2(\nabla ^{\ast
}A)(\nabla f)\right] df(\xi )+|\nabla (\xi f)|^{2} \\
&=&|\nabla (\xi f)|^{2}-df(\xi )\cdot \xi (\triangle f)+2df(\xi )\cdot
g(A,\nabla ^{2}f)-2df(\xi )(\nabla ^{\ast }A)(\nabla f),
\end{eqnarray*}%
which completes the proof of \eqref{e:vertical Bochner}.
\end{rmrk}

\section{Vanishing of the pseudohermitian torsion in the extremal case for $n\ge 2$}

In this section we prove Theorem \ref{t:A vansihes if div} which is one of our main results valid in dimension at least five, i.e., we assume $n\ge 2$. The assumptions in this section, unless noted otherwise, are that $M$ is a strictly pseudoconvex pseudohermitian CR manifold of dimension at least five  and $f$  satisfies \eqref{e:hessian}.

\subsection{Curvature in the extremal case}

We start with a calculation of the curvature tensor.  This is achieved by using \eqref{e:hessian}, \eqref{rho} and the Ricci
identities \eqref{e:ricci identities}. After some standard
calculations we obtain the following formula
\begin{multline}\label{eqc1}
R(Z,X,Y,\nabla f)=\Big[df(Z)g(X,Y)-df(X)g(Z,Y)\Big]+\nabla df(\xi ,Z)\omega
(X,Y)   \\
-\nabla df(\xi ,X)\omega (Z,Y)-2\nabla df(\xi ,Y)\omega (Z,X)+A(Z,\nabla
f)\omega (X,Y)-A(X,\nabla f)\omega (Z,Y).
\end{multline}%
Taking the traces in \eqref{eqc1} we obtain using \eqref{currrr}
\begin{equation}\label{eqc02}
\begin{aligned} & Ric(Z,\nabla f)=(2n-1)df(Z)-A(JZ,\nabla f)-3\nabla
df(\xi,JZ)\\ & Ric(JZ,J\nabla f)=R(JZ,Je_a,e_a,\nabla
f)=df(Z)-(2n-1)A(JZ,\nabla f)-(2n+1)\nabla df(\xi,JZ). \end{aligned}
\end{equation}
We note that the above derivation of \eqref{eqc1} and \eqref{eqc02} holds also when $n=1$.
\subsection{The vertical parts of the Hessian in the extremal case}

Subtracting the equations in \eqref{eqc02} and using \eqref{torric} we
obtain
\begin{equation}
\nabla df(\xi ,JZ)=-df(Z)+A(JZ,\nabla f)  \label{eqc01}
\end{equation}%
after dividing by $n-1$ since $n>1$. Equation \eqref{eqc01} and the Ricci identity yield
\begin{equation}
\label{e:vhessian}
 \nabla ^{2}f(\xi ,Y)=df(JY)+A(Y,\nabla f),\qquad \nabla ^{2}f(Y,\xi
)=df(JY)+2A(Y,\nabla f).
\end{equation}
At this point we have not yet determined $\xi^2 f$, but this will be achieved in Lemma \ref{l:ntor1}.
\subsection{The relation between $A$ and $A\protect\nabla f$.}

\begin{lemma}
\label{l:R1 part} Let $M$ be a strictly pseudoconvex pseudohermitian CR
manifold of dimension $2n+1\geq 5$. If $f$ is a function satisfying %
\eqref{e:hessian}, then we have the following identity
\begin{equation}  \label{e:norms dfA vs Adf}
|\nabla f|^{2}|A|^{2}=2|A\nabla f|^{2}.
\end{equation}
\end{lemma}

\begin{proof}
From \eqref{eqc1} we have
\begin{multline}
R(X,Y,Z,\nabla f)=df(X)g(Y,Z)-df(Y)g(X,Z)+\omega (Y,Z)[df(JX)+A(X,\nabla f)]
\label{e:eqc1} \\
-\omega (X,Z)[df(JY)+A(Y,\nabla f)]-2\omega (X,Y)[df(JZ)+A(Z,\nabla
f)]+A(X,\nabla f)\omega (Y,Z)-A(Y,\nabla f)\omega (X,Z) \\
=df(X)g(Y,Z)-df(Y)g(X,Z)+df(JX)\omega (Y,Z)-df(JY)\omega (X,Z)-2df(JZ)\omega
(X,Y) \\
-2\omega (X,Y)A(Z,\nabla f)+2A(X,\nabla f)\omega (Y,Z)-2A(Y,\nabla f)\omega
(X,Z),
\end{multline}
therefore
\begin{multline} \label{e:currrr}
R(X,Y,Z,\nabla f)-R(JX,JY,Z,\nabla f) \\
=2\omega (Y,Z)A(X,\nabla f)-2\omega (X,Z)A(Y,\nabla f)+2g(Y,Z)A(JX,\nabla
f)-2g(X,Z)A(JY,\nabla f).
\end{multline}
On the other hand, from \eqref{currrr} we have
\begin{multline}  \label{e:currrr1}
R(X,Y,Z,\nabla f)-R(JX,JY,Z,\nabla f) \\
=-2g(X,Z)A(Y,J\nabla f)-2g(Y,\nabla f)A(X,JZ)+2g(Y,Z)A(X,J\nabla f) \\
+2g(X,\nabla f)A(Y,JZ)-2\omega (X,Z)A(Y,\nabla f)-2\omega (Y,\nabla
f)A(X,Z)+2\omega (Y,Z)A(X,\nabla f)+2\omega (X,\nabla f)A(Y,Z).
\end{multline}
Comparing equations \eqref{e:currrr}, \eqref{e:currrr1} and taking into
account the type of $A$, $A(JX,Y)=A(X,JY)$, we come to
\begin{multline}  \label{e:R1 part 1}
0 =-2g(Y,\nabla f)A(X,JZ)+2g(X,\nabla f)A(Y,JZ)-2\omega (Y,\nabla
f)A(X,Z)+2\omega (X,\nabla f)A(Y,Z) \\
=-2df(Y)A(X,JZ)+2df(X)A(Y,JZ)-2df(JY)A(X,Z)+2df(JX)A(Y,Z).
\end{multline}%
Taking $X=\nabla f$ in \eqref{e:R1 part 1} we obtain the identity
\begin{eqnarray*}
|\nabla f|^{2}A(Y,Z) &=&df(Y)A(\nabla f,Z)-df(JY)A(\nabla f,JZ)
\end{eqnarray*}%
which proves \eqref{e:norms dfA vs Adf}.
\end{proof}

\subsection{The vertical derivative of an extremal eigenfunction}

\begin{lemma}
\label{l:D3f} Let $M$ be a strictly pseudoconvex pseudohermitian CR manifold
of dimension $2n+1\geq 5$. If $f$ is an eigenfunction of the sub-Laplacian
satisfying \eqref{e:hessian}, then the following formula for the third
covariant derivative holds true
\begin{equation}  \label{e:D3f extremal}
\nabla ^{3}f(X,Y,\xi ) =-df(\xi )g(X,Y)+f\omega (X,Y)-2fA(X,Y)-2df(\xi
)A(JX,Y) +2(\nabla_X A)(Y,\nabla f).
\end{equation}
\end{lemma}

\begin{proof}
Differentiating the last identity in \eqref{e:vhessian} we obtain
\begin{equation*}
\nabla ^{3}f(X,Y,\xi )=\nabla ^{2}f(X,JY)+2(\nabla_X A)(Y,\nabla
f)+2A(Y,\nabla _{X}(\nabla f)).
\end{equation*}%
Now, invoking \eqref{e:hessian} gives the desired formula.
\end{proof}

\begin{lemma}
\label{l:xi2f}Let $M$ be a strictly pseudoconvex pseudohermitian CR manifold
of dimension $2n+1\geq 5$. If $f$ is an eigenfunction satisfying %
\eqref{e:hessian} then we have
\begin{equation}
\nabla ^{2}f(\xi ,\xi )=\xi ^{2}f=-f-\frac1n(\nabla_{e_a} A)(e_{a},J\nabla
f).  \label{e:xi2f}
\end{equation}
\end{lemma}

\begin{proof}
Comparing equations \eqref{e:D3f extremal} and \eqref{e:D3f extremal bis} we
obtain the identity
\begin{multline*}
-(\xi ^{2}f)\omega (X,Y)+(\nabla A)(X,Y,\nabla f)+(\nabla A)(Y,X,\nabla
f)-(\nabla A)(\nabla f,X,Y) \\
=f\omega (X,Y)-2df(\xi )A(JX,Y)+2\nabla A(X,Y,\nabla f).
\end{multline*}
Taking a trace we get \eqref{e:xi2f}.
\end{proof}

\begin{lemma}
\label{l:ntor1} Let $M$ be a strictly pseudoconvex pseudohermitian CR
manifold of dimension $2n+1\geq 5$. If the pseudohermitian torsion is
divergence-free, $(\nabla^* A)(X)=0$, and $f$ is an eigenfunction satisfying %
\eqref{e:hessian} then the next formulas hold true
\begin{equation}  \label{ntor1}
\nabla^2f(\xi,\xi)=\xi^2 f=-f,\qquad (\nabla_X A) (Y,\nabla
f)=fA(X,Y)+df(\xi)A(X,JY).
\end{equation}
\end{lemma}

\begin{proof}
The first part follows immediately from Lemma \ref{l:xi2f}. However, both
parts can be seen as follows. Using \eqref{xii4} in \eqref{e:D3f extremal}
we obtain the identity
\begin{equation}
(\nabla _{X}A)(Y,\nabla f)=fA(X,Y)+df(\xi )A(X,JY)-\frac{1}{2}\Big[\nabla
^{2}f(\xi ,\xi )+f\Big]\omega (X,Y).  \label{ntor2}
\end{equation}%
Equation \eqref{ntor2} yields
\begin{multline}
(\nabla _{X}A)(JY,\nabla f)=(\nabla _{X}A)(Y,J\nabla f)= fA(X,JY)-df(\xi
)A(X,Y)-\frac{1}{2}\Big[\nabla ^{2}f(\xi ,\xi )+f\Big]g(X,Y),  \label{ntor3}
\end{multline}%
where we used \eqref{tortrace}. Taking the trace of \eqref{ntor3} using the
fact that the pseudohermitian torsion is both divergence-free and trace-free
we obtain the proof of the lemma.
\end{proof}

\subsection{The elliptic eigenvalue problem.}
A consequence of the above Lemma \ref{l:ntor1} is the following fact, which
plays a crucial role in resolving Conjecture \ref{conj1} in the vanishing
torsion case by allowing the reduction to the Riemannian Obata theorem. Furthermore, the elliptic equation satisfied by an extremal eigenfunction shows that $|\nabla f|\not=0$, hence $df\not=0$, in a dense set since $f\not= const$.

\begin{cor}\label{c:riem-eigen-fn} Let $M$ be a strictly pseudoconvex pseudohermitian
CR manifold of dimension $2n+1\geq 5$. If the pseudohermitian torsion of $M$
is divergence-free, $(\nabla^* A)(X)=0$, and $f$ is an eigenfunction
satisfying \eqref{e:hessian}, then $f$ is an eigenfunction of the (positive)
Riemannian Laplacian $\triangle^h $ associated to  the Riemannian
metric
\begin{equation}  \label{e:riem extension}
h =  g + \theta^2
\end{equation}
on $M$,  satisfying
\begin{equation}  \label{e:riem-eigen-fn}
\triangle^h f =  ( 2n+1) f.
\end{equation}
\end{cor}

\begin{proof}
We denote by $D$ the Levi-Civita connection of $h$. With respect to the
local orthonormal basis $e_a,  \xi$, $a=1,\dots,2n$, for any
smooth function $f$ we have
\begin{equation*}
-\triangle^h f = h(D_{e_a} (Df),e_a)+ h(D_{\xi}(Df),\xi).
\end{equation*}
Since $g$ and $\theta$ are parallel for the Tanaka-Webster connection $%
\nabla $ it follows $\nabla h =0$, which allows us to find the relation
between the two connections:
\begin{equation}  \label{lcbi}
h(\nabla_AB,C)=h(D_AB,C)+\frac12\Big[ h(T(A,B),C)-h(T(B,C),A)+h(T(C,A),B)%
\Big], \qquad A,\, B,\, C\in \mathcal{T}(M).
\end{equation}
Therefore we have
\begin{equation*}
h(\nabla_{e_a}B,{e_a})=h(D_{e_a}B,{e_a}) -h(T(B,{e_a}),{e_a}), \qquad
h(\nabla_{\xi}B,{\xi})=h(D_{\xi}B,{\xi}) -h(T(B,{\xi}),{\xi}), \qquad B \in
\mathcal{T}(M).
\end{equation*}
With this relation in mind, taking into account the properties of the
torsion, the formula for the Laplacian reduces to
\begin{equation}\label{obsa}
-\triangle^h f =-\triangle f+ h(\nabla_{\xi}(Df),\xi)=-\triangle
f+ (\nabla^2 f)(\xi,\xi)=-\triangle f+ (\xi^2 f).
\end{equation}
In particular, if $f$ satisfies \eqref{e:hessian} then we have $\triangle
f=2nf$, while Lemma \ref{l:ntor1} gives $\xi^2 \, f=-f$, hence the claimed
identity.
\end{proof}


\subsection{Vanishing of the pseudohermitian torsion in the case $n>1$.}

\begin{lemma}
\label{l:ntor01} Let $M$ be a strictly pseudoconvex pseudohermitian CR
manifold of dimension $2n+1\geq 5$. If the pseudohermitian torsion is
divergence-free, $(\nabla_{e_a}A)(e_a,X)=0$, and $f$ is an eigenfunction
satisfying \eqref{e:hessian} then
\begin{equation}  \label{tordf}
A\nabla f=0.
\end{equation}
\end{lemma}

\begin{proof}
From Lemma \ref{c:xi f} it follows that $\xi f$ is also an extremal first
eigenfunction. Therefore, the second equation in \eqref{ntor1} of Lemma \ref%
{l:ntor1} applied to $\xi f$ gives
\begin{equation}
(\nabla _{X}A)(Y,\nabla (\xi f))=(\xi f)A(X,Y)+(\xi^2 f)A(X,JY).
\label{ntor4}
\end{equation}%
A substitution of the second equality of \eqref{e:vhessian} in \eqref{ntor4}
shows
\begin{equation*}
-(\nabla _{X}A)(Y,J\nabla f)+2(\nabla _{X}A)(Y,A\nabla f)=(\xi f)A(X,Y)+(\xi
^{2}f)A(X,JY).
\end{equation*}%
Noting that $A(Y,JX)=A(JY,X)$ by the last equality of \eqref{tortrace}, the
above identity together with Lemma \ref{l:ntor1} give
\begin{equation}
2(\nabla _{X}A)(Y,A\nabla f)=(\nabla _{X}A)(JY,\nabla f)+df(\xi
)A(X,Y)-fA(X,JY)=0.  \label{ntor7}
\end{equation}%
Therefore, using the symmetry of $A$ and \eqref{ntor1} of Lemma \ref{l:ntor1}
we have
\begin{equation}  \label{ntor8}
0=(\nabla _{X}A)(\nabla f,A\nabla f)=(\nabla _{X}A)(A\nabla f,\nabla
f)=fA(X,A\nabla f)+df(\xi )A(JX,A\nabla f).
\end{equation}%
Replacing $X$ with $JX$ in \eqref{ntor8} yields
\begin{equation}
0=fA(JX,A\nabla f)-df(\xi )A(X,A\nabla f).  \label{ntor9}
\end{equation}%
From Corollary \ref{c:riem-eigen-fn} it follows that $f$ is an eigenfunction
for a Riemannian Laplacian, hence it cannot vanish on an open set unless $%
f\equiv 0$. Since by assumption $f$ is non-trivial, equations \eqref{ntor8}
and \eqref{ntor9}, taking into account $f^{2}+(df(\xi ))^{2}\not=0$, a.e.,
imply
\begin{equation}
A(X,A\nabla f)=0,\quad \text{ i.e, }\quad A\nabla f=0.  \label{ntor10}
\end{equation}
\end{proof}

\begin{thrm}
\label{t:A vansihes if div} Let $M$ be a strictly pseudoconvex
pseudohermitian CR manifold of dimension $2n+1\geq 5$. If the
pseudohermitian torsion is divergence-free, $(\nabla_{e_a}A)(e_a,Z)=0$ and $%
f $ is an eigenfunction satisfying \eqref{e:hessian} then the
pseudohermitian torsion vanishes, $A=0$.
\end{thrm}

\begin{proof}
Since the pseudohermitian torsion is divergence-free we have Lemma~\ref%
{l:ntor01}. Now, Lemma \ref{l:R1 part} shows that $A=0$.
\end{proof}

\section{Vanishing of the pseudohermitian torsion in the extremal three dimensional case}\label{s:3d case}

In this section we prove our first main result in dimension three. We shall assume, unless explicitly stated otherwise, that $M$ is a  compact strictly pseudoconvex pseudohermitian CR manifold of dimension three for which \eqref{condm} holds and $f$ is a smooth function on $M$ satisfying \eqref{e:hessian}. In particular, we have done the normalization, if necessary, so that \eqref{condm} holds with $k_0=4$. Since the horizontal space is two dimensional we can use $\gr$, $J\gr$ as a basis at the points where $|\gr|\not=0$.  In fact, similarly to the higher dimensional case, we have $|\gr|\not=0$ almost everywhere. This follows from Lemma~\ref{l:riem-eqn in 3D} showing that $f$ satisfies a certain elliptic equation which implies that $f$ cannot vanish on any open set since otherwise $f\equiv 0$ which is a contradiction.

\subsection{The elliptic value problem in dimension three}

In dimension three we have the following result.

\begin{lemma}\label{l:riem-eqn in 3D}
Let $M$ be a strictly pseudoconvex pseudohermitian
CR manifold of dimension three and pseudohermitian scalar curvature $S$. If the pseudohermitian torsion of $M$ is divergence-free, $(\nabla^* A)(X)=0$, and
$f$ is an eigenfunction satisfying \eqref{e:hessian}, then $f$ satisfies the following elliptic equation
\begin{equation}  \label{e:riem-eqn in 3D}
\triangle^h f = \left ( 2+ \frac{S-2}{6} \right ) f - \frac{{1}}{12}g(\gr,\nabla S)
\end{equation}
involving the (positive) Riemannian Laplacian $\triangle^h $ associated to  the Riemannian metric
\begin{equation}  \label{e:riem extension2}
h= g + \theta^2.
\end{equation}
\end{lemma}

\begin{proof}
We start by proving the  identity
\begin{equation}\label{e:xi2f 3D}
6\xi ^{2}f=-(S-2)f+ \frac 12g(\gr,\nabla S).
\end{equation}%
On one hand from \eqref{xi1} applied to the function $\xi f$ we have
\begin{equation*}
\nabla ^{2}f(Je_{a},e_{a},\xi )=2\xi ^{2}f.
\end{equation*}%
On the other hand we use \eqref{eqc02} to find
\begin{equation*}
3\nabla ^{2}f(Z,\xi )=-df(JZ)+2A(Z,\nabla f)-\rho (Z,\nabla f)
\end{equation*}
taking into account \eqref{rid}.
Now, differentiating the above identity, then taking a trace, after which using the
formula for the horizontal Hessian \eqref{e:hessian} gives
\begin{multline*}
6\xi ^{2}f=3\nabla ^{2}f(Je_{a},e_{a},\xi )=\triangle f +2 (\nabla^* A)(J\nabla f) + 2g(\nabla^2 f, A)-(\nabla^* \rho) (J\gr)-\nabla^2 f( Je_a, \rho e_a)\\
= 2f-(\nabla^* \rho) (J\gr)-\left [ -f\rho(e_a,Je_a)+(\xi f)\rho(e_a,e_a)\right ]=2f-fS+ \frac 12 dS(\gr )
\end{multline*}%
after using \eqref{div} and \eqref{rid} for the last equality.
This proves \eqref{e:xi2f 3D}.

At this point we invoke the proof of  Corollary \ref{c:riem-eigen-fn}. In fact, a substituition of  the above found expression for $\xi ^{2}f$ into \eqref{obsa}, which holds also in dimension three, gives \eqref{e:riem-eqn in 3D}.
\end{proof}

\subsection{Vanishing of the pseudohermitian torsion in the case $n=1$.}\label{ss:3D}

 Since $\nabla f,J\nabla f$ is a basis (not an orthonormal one!) of the horizontal space almost everywhere the vanishing of the pseudohermitian torsion, $A=0$, is implied by $A(\gr,\gr)=A(J\gr,\gr)=0$. We turn to the main result of this section.

\begin{thrm}\label{t:A vansihes if div 3D}
Let $M$ be a compact strictly pseudoconvex
pseudohermitian CR manifold of dimension three for which the Lichnerowicz condition \eqref{condm} holds. If the
pseudohermitian torsion is divergence-free, $(\nabla_{e_a}A)(e_a,Z)=0$ and $%
f $ is an eigenfunction satisfying \eqref{e:hessian} then the
pseudohermitian torsion vanishes, $A=0$, and the pseudohermitian scalar curvature is a constant, $S=8$.
\end{thrm}

\begin{proof}
Equation \eqref{eq14} yields
\begin{equation}
Ric(\nabla f,\nabla f)=\frac{S}{2}|\nabla f|^{2}=4|\nabla f|^{2}-4A(J\nabla
f,\nabla f).  \label{n11}
\end{equation}%
Setting $Z=\nabla f$  in \eqref{eqc02} and using \eqref{n11} we have
\begin{equation}
\nabla ^{2}f(\xi ,J\nabla f)=-|\nabla f|^{2}+A(J\nabla f,\nabla f)=-\frac{1%
}{4}Ric(\nabla f,\nabla f)=-\frac{S}{8}|\nabla f|^{2}.  \label{n12}
\end{equation}%
Taking $Z=J\nabla f$ in \eqref{eqc02} and using that in dimension three $%
Ric(X,JX)=0$ give
\begin{equation}
\nabla ^{2}f(\xi ,\nabla f)=-\frac{1}{3}A(\nabla f,\nabla f),\qquad \nabla
^{2}f(\nabla f,\xi )=\frac{2}{3}A(\nabla f,\nabla f).  \label{n13}
\end{equation}%
By Lemma \ref{c:xi f} we know that $\xi f$ is also an extremal eigenfunction, hence we have
\begin{equation}
Ric(\nabla (\xi f),\nabla (\xi f))=4|\nabla (\xi f)|^{2}-4A(\nabla (\xi
f),J\nabla (\xi f)).  \label{n=11xi}
\end{equation}%
Since $\nabla f$ and $J\nabla f$ are orthogonal we have
\begin{equation*}
|\nabla f|^{2}\nabla (\xi f)=\nabla ^{2}f(\nabla f,\xi )\nabla f+\nabla
^{2}f(J\nabla f,\xi )J\nabla f.
\end{equation*}%
Therefore, we have the following identities
\begin{equation}\label{xi2}
|\nabla f|^{4}|\, \nabla (\xi f)|^{2}\ =\ \Big[\ \Big(\nabla ^{2}f(\nabla f,\xi )\Big)%
^{2}\ +\ \Big(\nabla ^{2}f(J\nabla f,\xi )\Big)^{2}\ \Big]\, |\nabla f|^{2},
\end{equation}%
\begin{equation}\label{xi3}
|\nabla f|^{4}\,Ric(\nabla (\xi f),\nabla (\xi f))\ =\ \Big[\ \Big(\, \nabla
^{2}f(\nabla f,\xi )\,\Big)^{2}\ +\ \Big(\,\nabla ^{2}f(J\nabla f,\xi )\,\Big)^{2}\ \Big]%
\, Ric(\nabla f,\nabla f),
\end{equation}%
\begin{multline}
|\nabla f|^{4}\,A(\nabla (\xi f),J\nabla (\xi f)\ =\ \Big[\ \Big( \,\nabla ^{2}f(\nabla
f,\xi \,)\Big)^{2} \ -\ \Big(\, \nabla ^{2}f(J\nabla f,\xi )\, \Big)^{2}\ \Big]\,A(\nabla
f,J\nabla f)  \label{xi4} \\
-\ 2\Big (\, \nabla ^{2}f(\nabla f,\xi )\, \Big )\, \Big(\,\nabla ^{2}f(J\nabla f,\xi )\, \Big )A(\nabla f,\nabla f).
\end{multline}%
Substituting \eqref{xi4}, \eqref{xi3}, \eqref{xi2} in \eqref{n=11xi} and
using \eqref{n11} we obtain
\begin{equation}
|\nabla f|^{4}\left\{ \Big(\nabla ^{2}f(J\nabla f,\xi )\Big)^{2}A(\nabla
f,J\nabla f)\ + \ \nabla ^{2}f(\nabla f,\xi )\, \nabla ^{2}f(J\nabla f,\xi
)\,A(\nabla f,\nabla f)\right\}\ = \ 0.  \label{xi5}
\end{equation}%
Assumption \eqref{condm} implies  that in our case the pseudohermitian scalar curvature satisfies the inequality $S\geq 8$, hence \eqref{n11} yields
\begin{equation}
A(J\nabla f,\nabla f)=\big(1-\frac{S}{8}\big)|\nabla f|^{2}\leq 0.  \label{n=11x}
\end{equation}%
Equation \eqref{n12}, the Ricci identities and \eqref{n=11x} imply the inequality
\begin{equation}
\nabla ^{2}f(J\nabla f,\xi )=\nabla ^{2}f(\xi ,J\nabla f)+A(J\nabla f,\nabla
f)=A(J\nabla f,\nabla f)-\frac{S}{8}|\nabla f|^{2}\leq 0.  \label{n=12x}
\end{equation}%
Taking into account \eqref{n=12x} and \eqref{n13} we obtain from \eqref{xi5}
\begin{equation}
\nabla ^{2}f(J\nabla f,\xi )\, A(\nabla f,J\nabla f)\ +\ \frac{2}{3}\Big(A(\nabla
f,\nabla f)\Big)^{2}\ =\ 0.  \label{xi6}
\end{equation}%
The first term in \eqref{xi6} is nonnegative from
\eqref{n=11x} and \eqref{n=12x}. Therefore, we conclude
\begin{equation*}
A(\nabla f,J\nabla f)=A(\nabla f,\nabla f)=0,\quad i.e.\quad A=0.
\end{equation*}
The claim for the pseudohermitian scalar curvature follows for example from \eqref{n11}.
\end{proof}
An immediate corollary from Lemma \ref{l:riem-eqn in 3D} and Theorem \ref{t:A vansihes if div 3D} is the fact that on a strictly pseudoconvex pseudohermitian CR manifold of dimension three with a divergence-free pseudohermitian torsion every extremal eigenfunction is an eigenfunction of the Riemannian Laplacians \eqref{e:riem-eigen-fn}. In other words, Corollary \ref{c:riem-eigen-fn} is valid for $n\geq 1$.

\section{Proof of Theorem~\ref{main11}}
We prove Theorem~\ref{main11} by a reduction to the corresponding Riemannian  Obata theorem on a complete Riemannian manifold. In fact, we shall show that the Riemannian Hessian computed with respect to the Levi-Civita connection $D$ of the metric $h$ defined in \eqref{e:riem extension} satisfies \eqref{clasob}
and then apply the Obata theorem \cite{O3} to conclude that $(M,h)$ is isometric to the unit sphere.

For $n>1$ by Theorem \ref{t:A vansihes if div} it follows that $A=0$. Therefore,  \eqref{e:vhessian} and \eqref{ntor1} imply
\begin{equation}\label{e:vhessianc}
 \nabla ^{2}f(\xi ,Y)=\nabla ^{2}f(Y,\xi
)=df(JY), \qquad \xi^2f=-f.
\end{equation}

We show that \eqref{e:vhessianc} also holds in dimension three when the pseudohermitian torsion vanishes. In the three dimensional case we have $Ric(X,Y)=\frac{S}2g(X,Y)$. After a substitution of this  equality in \eqref{eqc02}, taking into account $A=0$, we obtain
\begin{equation}\label{hes3}
\bi^2f(\xi,Z)=\bi^2f(Z,\xi)=\frac{(S-2)}6df(JZ).
\end{equation}
Differentiating \eqref{hes3} and using \eqref{e:hessian} we find
\begin{equation}\label{hes31}
\bi^3f(Y,Z,\xi)=\frac16\Big[dS(Y)df(JZ)+(S-2)f\omega(Y,Z)-(S-2)df(\xi)g(Y,Z)\Big].
\end{equation}
On the other hand, setting $A=0$ in \eqref{e:D3f extremal bis}, we have
\begin{equation}\label{hes32}
\bi^3f(Y,Z,\xi)=-df(\xi)g(Y,Z)-(\xi^2f)\omega(Y,Z).
\end{equation}
In particular, the function $\xi f$ also satisfies \eqref{e:hessian}. Therefore using Lemma~\ref{l:riem-eqn in 3D} it follows that either $\xi f\equiv 0$ or $\xi f\not=0$ almost everywhere. In the first case it follows $\nabla f=0$ taking into account \eqref{hes3}, hence $f\equiv 0$, which is not possible by assumption. Thus, the second case holds, i.e., $\xi f\not=0$ almost everywhere.

Continuing our calculation, we note that  \eqref{hes31} and \eqref{hes32} give
\begin{equation}\label{hes33}
\frac{S-8}6df(\xi)g(Y,Z)-\Big(\xi^2f+\frac{S-2}6\Big)\omega(Y,Z)-\frac16dS(Y)df(JZ)=0,
\end{equation}
which implies
\begin{equation*}
\frac{S-8}3df(\xi)|\gr|^2=0.
\end{equation*}
Thus, the pseudohermitian scalar curvature is constant, $S=8$, invoking again Lemma~\ref{l:riem-eqn in 3D}. Equation  \eqref{hes33} reduces then to
\begin{equation}\label{hes34}
\Big(\xi^2f+\frac{S-2}6\Big)\omega(Y,Z)=0
\end{equation}
since $dS=0$. The equality \eqref{hes34} yields
$$\xi^2f=-f,$$
which together with \eqref{hes3} and $S=8$ imply the validity of \eqref{e:vhessianc} also in dimension three.

Finally, we use  \cite[Lemma~1.3]{DT} which gives the relation between $D$ and $\bi$. In the case $A=0$ we obtain
\begin{equation}\label{wh}
D_BC=\bi_BC+\theta(B)JC+\theta(C)JB-\omega(B,C)\xi,
\end{equation}
where $J$ is extended with $J\xi=0$. Notice that $\Omega$ in \cite{DT} is equal to $-\omega$. Using \eqref{wh} together with \eqref{e:hessian} and \eqref{e:vhessianc} we calculate that
\eqref{clasob} holds. The proof of Theorem~\ref{main11} is complete.

\section{Some examples}
In this short section we give some cases in which Theorem \ref{t:A vansihes if div} can be applied.
\begin{cor}
\label{t:A vansihes if vert} Let $M$ be a strictly pseudoconvex
pseudohermitian CR manifold of dimension $2n+1\geq 3$. If the vertical part
of the Ricci tensor of the Tanaka-Webster connection vanishes, $Ric(\xi,X)=0$%
, or the vertical part of the Ricci 2-form of the Tanaka-Webster connection
is zero, $\rho(\xi,X)=0$ and $f$ is an eigenfunction satisfying %
\eqref{e:hessian} then the pseudohermitian torsion vanishes, $A=0$.

In particular, if the vertical curvature of the Tanaka-Webster connection is
zero, and $f$ is an eigenfunction satisfying \eqref{e:hessian} then the
pseudohermitian torsion vanishes, $A=0$.
\end{cor}

\begin{proof}
Suppose that either
\begin{equation*}
0=Ric(\xi,X)=R(\xi ,e_a,e_{a},X)=R(\xi ,X,e_{a},Z)=0
\end{equation*}
or
\begin{equation*}
0=\rho(\xi,X)=R(\xi ,X,e_{a},Je_a)=R(\xi ,X,e_{a},Z)=0.
\end{equation*}
In both cases \eqref{currr} yields
\begin{equation}
(\nabla _{e_{a}}A)(e_{a},X)=0, .  \label{vcurv2}
\end{equation}%
since the pseudohermitian torsion $A$ is trace-free, symmetric and of type
(0,2)+(2,0) with respect to $J$. Hence, the pseudohermitian torsion is
divergence-free and the proof follows from Theorem \ref{t:A vansihes if div}.
\end{proof}

The formula \eqref{currr} yields that the vertical curvature of the
Tanaka-Webster connection vanishes if and only if the pseudohermitian
torsion is a Codazzi tensor with respect to $\nabla$,
\begin{equation}
(\nabla _{Y}A)(Z,X)=(\nabla _{Z}A)(Y,X).  \label{vcurv1}
\end{equation}
Due to Corollary~\ref{t:A vansihes if vert}, interesting examples where
Theorem~\ref{main2} and Theorem~\ref{main11} can be applied are provided by pseudohermitian
structures for which the pseudohermitian torsion $A$ is a Codazzi tensor.
This includes, of course, the case of a parallel pseudohermitian torsion
(contact $(\kappa,\mu)$-spaces), $(\nabla_XA)(Y,Z)=0$, for horizontal $X,Y$
and $Z$, see \cite{BoCho} and \cite{Pe}.

\section{Appendix}

\subsection{Some differential operators invariantly associated to the
pseudohermitian structure}

The purpose of this section is to record, for self-sufficiency,  some of the results of \cite{L1} and \cite{GL88} using real variables. For a full
description and further references of conformally invariant operators see \cite{Gra83ab}%
, \cite{GoGr05} and, in particular, \cite{Hi93} for the three dimensional
case.

\begin{dfn}[\protect\cite{L1,GL88}]
\label{d:BPC def} a) Let $B(X,Y)$ be the $(1,1)$ component of the horizontal
Hessian $(\nabla ^{2}f)(X,Y),$
\begin{equation}  \label{e:Bdef}
B(X,Y)\equiv B[f](X,Y)=(\nabla ^{2}f)_{[1]}(X,Y)=\frac{1}{2}\left[ (\nabla
^{2}f)(X,Y)+(\nabla ^{2}f)(JX,JY)\right]
\end{equation}%
and then define $B_{0}(X,Y)$ to be the completely traceless part of $B,$%
\begin{equation}  \label{e:B0def}
B_{0}(X,Y)\equiv B_{0}[f](X,Y)=B(X,Y)+\frac{\triangle f}{2n} g(X,Y)-\frac{1}{2n}%
g(\nabla^2f,\omega)\,\omega (X,Y).
\end{equation}
b) Given a function $f$ we define the one form,
\begin{equation}  \label{e:Pdef}
P(X)\equiv P_{f}(X)=\nabla ^{3}f(X,e_{b},e_{b})+\nabla
^{3}f(JX,e_{b},Je_{b})+4nA(X,J\nabla f)
\end{equation}%
and also  a fourth order differential operator  (the so called CR-Paneitz operator in \cite{Chi06}),
\begin{equation}  \label{e:Cdef}
Cf=-\nabla ^{\ast }P=(\nabla_{e_a} P)({e_a})=\nabla ^{4}f(e_a,e_a,e_{b},e_{b})+\nabla
^{4}f(e_a,Je_a,e_{b},Je_{b})-4n\nabla^* A(J\nabla f)-4n\,g(\nabla^2 f,JA).
\end{equation}
\end{dfn}

\begin{rmrk}
In the three dimensional case, the non-negativity condition %
\eqref{e:nonnegativeP} means
\begin{equation*}
\int_M f\cdot Cf Vol_{\theta}=-\int_MP_f(\gr) Vol_{\theta} \geq 0.
\end{equation*}
Note that this condition is a CR invariant since it is independent of the
choice of the contact form. This follows from the conformal invariance of $C$
proven in \cite{Hi93}. In the vanishing torsion case we have, up to a
multiplicative constant, $C=\Box_b\bar\Box_b$, where $\Box_b$ is the Kohn
Laplacian.
\end{rmrk}

From \cite{L1}, see also \cite{BF74} and \cite{Be80}, it is known that if $%
n\geq 2$, a function $f\in \mathcal{C}^3(M)$ is CR-pluriharmonic, i.e,
locally it is the real part of a CR holomorphic function, if and only if $%
B_0[f]=0$. In fact, as shown in \cite{GL88} with the help of the next
{Lemma} only one fourth-order equation $Cf=0$ suffices for $B_0[f]=0$ to
hold. Furthermore, the Paneitz operator is non-negative. When $n=1$ the
situation is more delicate. In the three dimensional case, CR-pluriharmonic
functions are characterized by the kernel of the third order operator $%
P[f]=0 $ \cite{L1}. However, the single equation $Cf=0$ is enough again
assuming the vanishing of the Webster torsion \cite{GL88}, see also \cite%
{Gra83ab}. On the other hand, \cite{CCC07} showed that if the torsion
vanishes the Paneitz operator is essentially positive, i.e., there is a
constant $\Lambda>0$ such that
\begin{equation*}
\int_M f\cdot Cf Vol_{\theta} \geq \Lambda\int_M f^2 Vol_{\theta} .
\end{equation*}
for all real smooth functions $f\in (Ker\, C)^{\perp}$, i.e., $\int_M f\cdot
\phi Vol_{\theta} =0$ if $C\phi=0$. In addition, the non-negativity of the
CR Paneitz operator is relevant in the embedding problem for a three
dimensional strictly pseudoconvex CR manifold. In the Sasakian case, it is
known  that $M$ is embeddable, \cite{Le92}, and the CR
Paneitz operator is nonnegative, see \cite{Chi06}, \cite{CCC07}. Furthermore, \cite{ChChY10} showed that if the pseudohermitian scalar curvature   of $M$ is positive and $C$ is non-negative, then $M$ is embeddable in some $\mathbb{C}^n$.

\begin{rmrk}\label{r:non-negative paneitz}
In the quaternionic contact setting the paper \cite{IMV} describes
differential operators and conformally invariant two forms whose kernel
contains the real parts of all anti-CRF functions. In particular, there is a
partial result characterizing real parts of anti-CRF functions. The "extra"
assumption concerns an analog of the $\partial\bar\partial$ lemma in the
quaternionic setting.
\end{rmrk}

We turn to one of the basic results relating the above defined operators.
\begin{lemma}[\protect\cite{GL88}]\label{l:GrLee}
On a he following identities holds true,
\begin{eqnarray*}
(\nabla_{e_{a}} B_{0})(e_{a},X) &=&\frac{n-1}{2n}P(X) \\
\int_M |B_0|^2\vol&=&-\frac{n-1}{2n}\int_M P(\gr)\vol.
\end{eqnarray*}
In particular, if $n>1$ the Paneitz operator is non-negative.
\end{lemma}

\begin{proof}
Taking into account Ricci's identity \eqref{e:ricci identities} we have
\begin{equation}\label{e:Ricci for P}
\begin{aligned}
\nabla ^{3}f(e_{a},e_{a},X)& =\nabla ^{3}f(X,e_{a},e_{a}) +Ric(X,\nabla
f)+4\nabla ^{2}f(\xi ,JX)+2A(JX,\xi ) \\
\nabla ^{3}f(e_{a},Je_{a},JX) &=\nabla ^{3}f(JX,Je_{a},e_{a})-2\rho
(JX,\nabla f)+Ric(JX,J\nabla f)-4n\nabla ^{2}f(\xi ,\ JX)\\  &\hskip3.7in -2A(JX,\gr ).
\end{aligned}
\end{equation}
Therefore, using \eqref{torric} and \eqref{xi1}, which implies $\nabla
^{2}f(JX,\xi )=\frac {1}{2n}\nabla ^{3}f(JX,e_{c},Je_{c})$, we have
\begin{multline}  \label{e:divB}
2\nabla _{e_a}(\nabla ^{2}f)_{[1]}(e_{a},X) =\nabla
^{3}f(X,e_{a},e_{a})-\nabla ^{3}f(JX,e_{a},Je_{a})+\frac{2(n-1)}{n}\nabla
^{3}f(JX,e_{a},Je_{a}) \\
+4(n-1)A(X,J\nabla f).
\end{multline}
The trace part of $B$ is computed as follows
\begin{multline}  \label{e:divtrB}
\nabla _{e_a}\left( \frac{1}{2n}\nabla ^{2}f(e_{c},e_{c})g-df(\xi )\omega
\right) (e_{a},X) =\frac{1}{2n}\nabla ^{3}f(X,e_{c},e_{c})+\nabla
^{2}f(JX,\xi ) \\
=\frac{1}{2n}\nabla ^{3}f(X,e_{a},e_{a})-\frac{1}{2n}\nabla
^{3}f(JX,e_{a},Je_{a})
\end{multline}
The above identities \eqref{e:divB} and \eqref{e:divtrB} imply
\begin{eqnarray*}
(\nabla B_{0})(e_{a},e_{a},X) &=&\frac{n-1}{2n}P(X),
\end{eqnarray*}
which completes the proof of the first formula. The second identity follows by an integration by parts.
\end{proof}
As was observed earlier in \cite{Chi06} and \cite{CCC07}, see also \cite{CaCh07}, if the pseudohermitian torsion vanishes then the Paneitz operator is non-negative also in dimension three.
{ For completeness since this is the case needed in Theorem \ref{main1} and because of its simplicity, we shall prove this fact for an eigenfunction of the sub-Laplacian, $\triangle f =\lambda f$ in the vanishing torsion case using an idea of \cite[Proposition 3.2]{GL88}. Indeed, since $A=0$ we have by the last Ricci identity \eqref{e:ricci identities} and \eqref{e:Pdef}}
\[
[\xi,\triangle]f=0,\hskip.5in P_{f}(X)=\nabla ^{3}f(X,e_{b},e_{b})+\nabla^{3}f(JX,e_{b},Je_{b}),
\]
therefore  with the help of \eqref{xi1} we have (recall \eqref{e:Cdef})
\begin{multline}
\int_M|P_f|^2\vol=\int_M \triangle f\cdot(Cf)\vol +\int_M \nabla^3f(Je_a,e_b,Je_b)\Big (\nabla^3(e_a,e_c,e_c)+\nabla^3f(Je_a,e_c,Je_c) \Big)\vol\\
=\lambda \int_M f\cdot(Cf)\vol -2n\int_M \nabla^2f(e_b,Je_b)\,[\xi,\triangle]f\vol=\lambda\int_M f\cdot(Cf)\vol.
\end{multline}
Thus, since $\triangle f$ is a non-negative operator, if the torsion vanishes and $f$ is an eigenfunction, then the Paneitz operator is non-negative on $f$,
\[
\int_M f\cdot(Cf)\vol\geq 0.
\]
{However, when dealing with a divergence-free pseudohermitian torsion  it is worth noting that the Paneitz operator is non-negative for any function satisfying \eqref{e:hessian}.
In the case of a divergence-free torsion the above identities hold taking into account the last Ricci identity \eqref{e:ricci identities}, the fact that $A$ belongs to the $(2,0)+(0,2)$ space, and definition \eqref{e:Cdef}.}

\subsection{Greenleaf's Bochner formula for the sub-Laplacian}

 The first
eigenvalue of the sub-Laplacian is the smallest positive constant in %
\eqref{eig} which we denote by $\lambda_1$.

\begin{thrm}[\protect\cite{Gr}]
On a strictly pseudoconvex pseudohermitian manifold of dimension $2n+1$, $n\geq 1$, the following Bochner-type identity holds
\begin{equation}  \label{bohh}
\frac12\triangle |\nabla f|^2=-g(\nabla(\triangle f),\nabla f)+Ric(\nabla
f,\nabla f)+2A(J\nabla f,\nabla f) +|\nabla df|^2+ 4\nabla df(\xi,J\nabla f).
\end{equation}
\end{thrm}

\begin{proof}
By definition we have
\begin{equation}  \label{boh1}
-\frac12\triangle |\nabla f|^2=\nabla^3f(e_a,e_a,e_b)
df(e_b)+\nabla^2f(e_a,e_b)\nabla^2f(e_a,e_b) =\nabla^3f(e_a,e_a,e_b) df(e_b)
+ |\nabla^2f|^2.
\end{equation}
To evaluate the first term in the right hand side of \eqref{boh1} we use the
Ricci identities \eqref{e:ricci identities}. Taking into account that
\begin{equation}  \label{par}
(\nabla_X T)(Y,Z)=0,
\end{equation}
applying successively the Ricci identities \eqref{e:ricci identities} and
also \eqref{par} we obtain
\begin{equation}  \label{boh3}
\nabla^3f(e_a,e_a,e_b)df(e_b)= -g(\nabla(\triangle f),\nabla f )+Ric(\nabla
f,\nabla f)+2A(J\nabla f,\nabla f)+4\nabla^2f(\xi,J\nabla f).
\end{equation}
A substitution of \eqref{boh3} in \eqref{boh1} completes the proof of %
\eqref{bohh}.
\end{proof}

We have to evaluate the last term of \eqref{bohh}. We recall the notation \eqref{comp} of
the two components  of the $U(n)$-invariant decomposition of the horizontal
Hessian $\nabla^2f$.
The first integral formula for the last term in \eqref{bohh} originally proved in \cite%
{Gr} follows.

\begin{lemma}[\cite{Gr}]
\label{gr2} On a compact strictly pseudoconvex pseudohermitian CR manifold
of dimension $2n+1$, $n\geq 1$, we have the identity
\begin{equation}  \label{2}
\int_M\nabla^2f(\xi,J\nabla f)Vol_{\theta}=-\int_M\Big[\frac{1}{2n}g(\nabla^2f,\omega)^2
+A(J\nabla f,\nabla f)\Big] Vol_{\theta}.
\end{equation}
\end{lemma}

\begin{proof}
Integrating \eqref{xi1} we compute
\begin{equation}  \label{vert1}
4n^2\int_M \left ( \xi f\right )^2 Vol_{\theta}=-2n\int_M
g(\nabla^2f,\omega)\cdot df(\xi)Vol_{\theta}.
\end{equation}
Let us consider the horizontal 1-form defined by
\begin{equation*}
D_2(X)=df(JX)df(\xi)
\end{equation*}
whose divergence is, taking into account the second formula of
\eqref{e:ricci
identities},
\begin{equation}  \label{vert2}
\nabla^*D_2=g(\nabla^2f,\omega)\cdot df(\xi)-\nabla^2f(\xi,J\nabla
f)-A(J\nabla f,\nabla f).
\end{equation}
Integrating \eqref{vert2} over $M$ and using \eqref{vert1} implies \eqref{2}
which completes the proof of the lemma.
\end{proof}

We shall need one more representation of the last term in \eqref{bohh}.

\begin{lemma}\label{gr3}
On a compact strictly pseudoconvex pseudohermitian CR manifold
of dimension $2n+1$, $n\geq 1$, we have the identity
\begin{equation*}
\nabla ^{2}f(\xi ,Z)=\frac{1}{2n}\nabla ^{3}f(Z,Je_a,e_a)-A(Z,\nabla f).
\end{equation*}
In addition,
\begin{equation*}
\int_{M}\nabla ^{2}f(\xi ,J\nabla f)Vol_{\theta }=\int_{M}-\frac{1}{2n}\left(
\triangle f\right) ^{2}+A(J\nabla f,\nabla f)-\frac{1}{2n}P(\nabla
f)Vol_{\theta }.
\end{equation*}
\end{lemma}

\begin{proof}
We compute using the first two Ricci identity in \eqref{e:ricci identities}
\begin{multline*}
2\nabla ^{3}f(Z,Je_a,e_a)=\nabla ^{3}f(Z,Je_a,e_a)-\nabla ^{3}f(Z,e_a,Je_a)=-2\omega(Je_a,e_a)\nabla^2 f(Z,\xi)\\=4n\left (  \nabla^2f(\xi,Z) +A(Z,\gr)\right),
\end{multline*}
which proves the first formula.  The second identity follows from the above formula, the definition of $P$, and an integration by parts.
\end{proof}

\subsection{The CR Lichneorwicz type theorem.}
At this point we are ready to give in details, using  real notation as in \cite%
{IVZ,IV2}, the known version of the Lichnerowicz' result on a
compact strictly pseudo-convex CR manifold \cite{Gr},
\cite{LL}, \cite{Chi06}.
\begin{thrm}[\protect\cite{Gr}, \protect\cite{LL}, \protect\cite{Chi06}]
\label{main1} Let $(M, \theta)$ be a compact strictly pseudoconvex
pseudohermitian CR manifold of dimension $2n+1$. Suppose there is a positive
constant $k_0$ such that the pseudohermitian Ricci curvature $Ric$ and the
pseudohermitian torsion $A$ satisfy the inequality
\begin{equation}  \label{condm-app}
Ric(X,X)+ 4A(X,JX)\geq k_0 g(X,X).
\end{equation}
a) If $n>1$, then any positive eigenvalue $\lambda$ of the sub-Laplacian $%
\triangle$ satisfies the inequality
\begin{equation*}
\lambda \ge \frac{n}{n+1}k_0.
\end{equation*}

b) If $n=1$ and the Paneitz operator is non-negative, i.e.,
\begin{equation}  \label{e:nonnegativeP}
-\int_M P_{f}(\nabla f)Vol_{\theta}\geq 0.
\end{equation}
where $f$ is a smooth function and
\begin{equation*}
P_{f}(X)=\nabla ^{3}f(X,e_{b},e_{b})+\nabla
^{3}f(JX,e_{b},Je_{b})+4A(X,J\nabla f),
\end{equation*}
then
\begin{equation*}
\lambda \ge \frac12k_0.
\end{equation*}
\end{thrm}
As well known the standard CR structure on the sphere achieves equality in this
inequality.   We
note that the assumption on the pseudohermitian Ricci curvature of the
Tanaka-Webster connection and the pseudohermitian torsion can be put in the
equivalent form
\begin{equation*}
Ric(X,X)+ 4A(X,JX)=\rho(JX,X)+2(n+1)A(JX,X)
\end{equation*}
using the pseudohermitian Ricci 2-form $\rho$ of the Tanaka-Webster
connection.
We turn to the proof of Theorem \ref{main1}.
\begin{proof}
Integrating the Bochner type formula \eqref{bohh} we obtain
\begin{multline}  \label{bohin}
0=\int_M\Big[-(\triangle f)^2+\left |(\nabla^2f)_{[1]}\right |^2 + \left
|(\nabla^2f)_{[-1]}\right |^2+Ric(\nabla f,\nabla f)+2A(J\nabla f,\nabla f)
+ 4\nabla^2f(\xi,J\nabla f)\Big] Vol_{\theta}.
\end{multline}
We use Lemma \ref{gr3} to represent the last term, which turns the above identity in the following
\begin{multline}  \label{e:bohin}
0=\int_M\Big[-(\triangle f)^2+\left |(\nabla^2f)_{[1]}\right |^2 + \left
|(\nabla^2f)_{[-1]}\right |^2+Ric(\nabla f,\nabla f)+6A(J\nabla f,\nabla f)\\
-\frac {2}{n}(\triangle f)^2 -\frac {2}{n}P(\gr)\Big]\vol.
\end{multline}
Lemma \ref{gr2} and Lemma \ref{gr3} give the following identity
\begin{equation}\label{e:Afrom2lemmas}
2\int_M  A(J\nabla f,\nabla f)\vol =\int_M\Big[ -\frac {1}{2n}g(\nabla^2 f,\omega)^2+\frac {1}{2n}(\triangle f)^2+\frac {1}{2n}P(\gr)\Big]\vol.
\end{equation}
A substitution of \eqref{e:Afrom2lemmas} this in \eqref{e:bohin} together with \eqref{condm} we obtain for an eigenfunction $\triangle f=\lambda f$ the inequality
\begin{multline}\label{e:bohin1}
0\geq\int_M\Big[-(\triangle f)^2+\left |(\nabla^2f)_{[1]}\right |^2 + \left
|(\nabla^2f)_{[-1]}\right |^2+k_0|\gr|^2-\frac {1}{2n}g(\nabla^2 f,\omega)^2
-\frac {3}{2n}(\triangle f)^2 -\frac {3}{2n}P(\gr)\Big]\vol\\
=\int_M\Big[\left (-\frac {n+1}{n}\lambda+k_0\right )|\gr|^2+\left |(\nabla^2f)_{[1]}\right |^2 -\frac {1}{2n}(\triangle f)^2-\frac {1}{2n}g(\nabla^2 f,\omega)^2+ \left
|(\nabla^2f)_{[-1]}\right |^2
 -\frac {3}{2n}P(\gr)\Big]\vol.
\end{multline}
A projection on the span of the orthonormal set $\Big \{\frac{1}{\sqrt{2n}}%
g,\ \frac{1}{\sqrt{2n}}w\Big\}$ in the $(1,1)$ space gives
\begin{equation}  \label{coshy3}
|(\nabla^2f)_{[1]}|^2\ge \frac {1}{2n}(\triangle f)^2+\frac {1}{2n}\left ( g
(\nabla^2f,\omega)\right )^2
\end{equation}
with equality iff
\begin{equation}  \label{equality hessian}
(\nabla^2f)_{[1]}= \frac {1}{2n}(\triangle f)\cdot g+\frac {1}{2n}g
(\nabla^2f,\omega)\cdot\omega.
\end{equation}
We obtain from \eqref{e:bohin1} taking into account \eqref{coshy3} the inequality

\begin{equation}  \label{e:obata ineq}
0\geq\int_M\Big[\left (-\frac {n+1}{n}\lambda+k_0\right )|\gr|^2
+ \left
|(\nabla^2f)_{[-1]}\right |^2 -\frac {3}{2n}P(\gr)\Big]\vol.
\end{equation}
This implies Greenleaf's inequality
\begin{equation*}
\lambda \ge \frac{n}{n+1}k_0.
\end{equation*}
taking into account Lemma \ref{l:GrLee}. This completes the proof of Theorem~\ref{main1}.
\end{proof}

\end{document}